\documentclass[10pt]{article}  
\usepackage{amsmath,amstext,amssymb,amsthm}
\title{On Gaussian Limits and Large Deviations for Queues Fed by High Intensity Randomly Scattered Traffic}
\author{Peter W. Glynn\\ 
Department of Management Science and Engineering\\
Stanford University\\
Stanford CA\\
glynn@stanford.edu
\and
Harsha Honnappa\\
School of Industrial Engineering\\
Purdue University\\
West Lafayette IN\\
honnappa@purdue.edu
}
\date{}
\newtheorem{theorem}{\sc Theorem}[section]

\newtheorem{lemma}{\sc Lemma}[section]

\newtheorem{corollary}{\sc Corollary}[section]

\newtheorem{example}{\sc Example}

\newtheorem{assumption}{\sc Assumption}
\newtheorem{proposition}{\sc Proposition}

\RequirePackage{fix-cm}
\usepackage{graphicx}

\providecommand{\KEYWORDS}[1]{\textbf{\textit{Index terms---}} #1}

\begin{document}
\maketitle

\begin{abstract}
	We study a single server FIFO queue that offers general service. Each of $n$ customers enter the queue at random time epochs that are independent and identically distributed. We call this the \emph{random scattering} traffic model, and the queueing model $RS/G/1$. We study the workload process associated with the queue in two different settings. First, we present Gaussian process approximations in a high intensity asymptotic scale and  characterize the transient distribution of the approximation. Second, we study the rare event paths of the workload by proving a large deviations principle in the same high intensity regime. We also obtain exact asymptotics for the Gaussian approximations developed prior. This analysis  significantly extends and simplifies recent work in \cite{HoJaWa2014} on uniform population acceleration asymptotics to the queue length and workload in the $RS/G/1$ queue.
\end{abstract}

\section{Introduction} \label{sec:intro}
A natural causal mechanism for describing arrivals to a service facility is one in which customers make independent decisions about when to access the system. In particular, suppose that each of $n$ customers decide to enter the queue at arrival times $T_1, T_2,\ldots, T_n$, where the $T_i$'s are independent and identically distributed (i.i.d.). For instance, subscribers to an online news website may have preferences as to when to read news items; such preferences are reflected in the common distribution $F$ of the $T_i$'s. A second service-related example involves modeling employee arrival times at a company cafeteria. We use the term \emph{random scattering} to describe this independent choice mechanism, and refer to the associated traffic model as a randomly scattered (RS) traffic model.

Such RS traffic is non-renewal. Furthermore, such traffic streams exhibit negative correlations that induce performance behavior that is qualitatively different from that associated with conventional renewal-type short-range dependent traffic. For example, the most likely ``rare event" path associated with a large workload exceedance is non-linear in general (rather than piecewise linear as in the conventional traffic model \cite{An1989}); see Theorem \ref{thm:workload-ldp} below. When considered in conjunction with the physical plausibility of RS traffic from a modeling viewpoint, this suggests that RS traffic may be worth adding to the set of traffic models commonly used as (for example) simulation inputs to queue system simulations. We further note that such models can be easily calibrated to real-world data. This stands in contrast to traffic models based on, for example, Markov-modulated Poisson processes.

This model was introduced by Gaver et al \cite{GaLePe1975}; that contribution includes a number of other real-world modeling applications. It is related to a ``binomial'' traffic model of Newell \cite{Ne1982}, and has been further studied by Louchard \cite{Lo1994} and Honnappa et al \cite{HoJaWa2014}. All of these papers, as does ours, focus on ``high
intensity'' RS models in which the population $n$ tends to infinity. Unlike our development, their focus is on the mathematical implications arising from imbalances associated with the system moving from extreme overload to extreme underload and vice versa.

In contrast, the first part of our paper considers moderate overload/underload settings in which the scattering distribution is close to uniform and the system has a constant service rate; see Theorem \ref{thm:workload-unbalanced}. The uniform distribution is of interest for two different reasons. First, the uniform distribution arises in connection with the Nash equilibrium to a non-atomic version of a ``concert arrival'' game in which customers seek to minimize a cost functional depending linearly on the waiting time and the number of people who have already arrived; see \cite{JaJuSh2011,HoJa2015}. Secondly, in some settings, one needs a distributed algorithm that well spreads out the user load over time (e.g., times at which remote sensors will upload data to a server). Uniform randomization of the traffic accomplishes this in a way that is robust to changes in the size of the user population. This moderate overload/underload setting simplifies the mathematical development, and leads to a reflected Gaussian process in which the Gaussian limit is the sum of a Brownian bridge and a Brownian motion process. Furthermore, we are able to compute the transient distribution of the process, despite the fact that the reflected process is non-Markov; see Section \ref{sec:distribution}. 


In addition to the above ``heavy traffic'' reflected Gaussian limit theory, to be found in Sections 2 through 4, we also study large deviations for the $RS/G/1$ queue in the high intensity regime in Section 5, and related large deviations for the limiting reflected Gaussian process in Section 6. The theory just described is developed in the setting of a single period model, and does not impose any ``moderate underload/overload'' hypothesis. Section 7 concludes the paper with a brief discussion of the generalization to equilibrium analysis of the multi-period version of the $RS/G/1$ model.

\section{Gaussian Limits for RS Traffic} \label{sec:gauss-limits}
We consider in this section a single period model in which $n$ customers with service times $V_1, V_2,\ldots, V_n$ enter the system at non-negative times $T_1,\ldots, T_n$. For $0 \leq s < t < \infty$, let $\Gamma_n(s,t]$ be the total work to enter the system within the interval $(s,t]$, so that

\begin{equation}
\Gamma_n(s,t] = \sum_{i=1}^n V_i \mathbf{1}_{\{T_i \in (s,t]\}},
\end{equation}
where $\mathbf{1}_{A}$ is the indicator random variable corresponding to the event $A$. For any two intervals $(s_1,s_2]$ and $(t_1,t_2]$ with $s_1 < t_1$ and $s_2 < t_2$,
\[
\begin{split}
	\text{Cov} \left(\Gamma_n(s_1,s_2], \Gamma_n(t_1,t_2]\right) = n &\left(EV_1^2 \mathbb{P}(T_1 \in (t_1,s_2]) \right)\\ &- (EV_1)^2 \mathbb{P}(T_1 \in (s_1,s_2]) \mathbb{P}(T_1 \in (t_1,t_2]),
\end{split}
\]
where $(t_1,s_2] = \emptyset$ if $t_1 > s_2$. Hence, if the intervals are non-overlapping (so that $s_2 < t_1$), the covariance is non-positive. The negative correlation arises as a consequence of the following implication of a fixed population size $n$: if $\Gamma_n(s_1,s_2]$ is unusually large, then $\Gamma_n(t_1,t_2]$ will be smaller than normal. In contrast, in typical renewal-type traffic having intensity $n$, the covariance between non-overlapping intervals is generally of order $o(n)$ as $n \to \infty$. Here, $o(a_n)$ represents a sequence for which $o(a_n)/a_n \to 0$ as $n \to \infty$.

Let 
\[
	F_n(\cdot) := \frac{1}{n} \sum_{i=1}^{n} \mathbf{1}_{\{T_i \leq \cdot\}}
\]
be the empirical distribution of the $T_i$'s. We turn first to almost sure (a.s.) functional Strong Laws of Large Numbers (FSLLN) for $F_n$ and $\Gamma_n(s,t]$.

\begin{proposition} \label{prop:FSLLN}
	If $EV_1 < \infty$, then
	\begin{equation}
		\sup_{t \geq 0} |F_n(t) - F(t)| \stackrel{a.s.}{\to} 0
	\end{equation}
	and
	\begin{equation}
		\sup_{0 \leq s < t <\infty} \left | \frac{1}{n} \Gamma_n(s,t] - EV (F(t) - F(s)) \right| \stackrel{a.s.}{\to} 0
	\end{equation}
	as $n \to \infty$.
\end{proposition}
The first result is the well known Glivenko-Cantelli Theorem; see \cite[Ch. 1]{Du2010}. The second result follows from the proof of the Glivenko-Cantelli Theorem, upon recognizing that $n^{-1}\Gamma_n(-\infty,\cdot]$ is a non-decreasing function with a convergent finite-valued limit $\Gamma_n(-\infty,\infty)$.

We next prove a functional limit theorem for $\Gamma_n$. In preparation for stating the results, let $B_1$ and $B_2$ be two independent standard Brownian motions and let $B^0_2(t) := B_2(t) - t B_2(1)$ for $0 \leq t \leq 1$. It is well known that $\mathbb{P} \left( B_2^0 \in \cdot \right) = \mathbb{P}\left(B_2 \in \cdot | B_2(1) = 0\right)$; see \cite{BoSa2012}. Hence, $B^0_2$ is equal in distribution to a Brownian bridge process. For $0 \leq t \leq 1$, set
\begin{equation} \label{eq:single-period-diffusion}
	Z := \sigma_V B_1 \circ F + EV_1 B_2^0 \circ F,
\end{equation}
where $\sigma^2_V = \text{Var}(V_1)$ and $\circ$ represents the composition of functions.

\begin{theorem} \label{thm:basics-fclt}
	If $EV_1^2 < \infty$, then 
	\begin{equation}
		\sqrt{n} \left(\frac{1}{n} \Gamma_n(-\infty,\cdot] - EV_1 F(\cdot)\right) \Rightarrow Z(\cdot),
	\end{equation}
	as $n \to \infty$ in $\mathcal{D}[0,\infty)$, where $\Rightarrow$ denotes weak convergence.
\end{theorem}

\begin{proof}
	For $0 \leq t \leq 1$, let 
	\[
		\chi_n(t) := \sqrt{n} \left( \frac{1}{n} \sum_{i=1}^{\lfloor{nt}\rfloor} V_i - EV_1 \right)
	\]
	and observe that $\sqrt{n} \left(\frac{1}{n} \Gamma_n(-\infty,\cdot] - EV_1 F(\cdot)\right)$
	\begin{eqnarray}
		 &\stackrel{D}{=}& \sqrt{n} \left(\frac{1}{n} \sum_{i=1}^{nF_n(\cdot)} V_i - EV_1 F(\cdot)\right)\\
		&=& \sqrt{n} \left(n^{-1} \sum_{i=1}^{nF_n(\cdot)} (V_i - EV_1)\right) + EV_1(F_n(\cdot) - F(\cdot) )\\
		&=& \chi_n \circ F_n (\cdot) + EV_1 \sqrt{n} \left(F_n(\cdot) - F(\cdot)\right),
	\end{eqnarray}
	where $\stackrel{D}{=}$ denotes ``denotes equality in distribution," and we utilized the i.i.d. structure of the $V_i$'s
 for the equality in distribution step.
 
 Now, standard results in weak convergence imply that 
 \[
	 \chi_n \Rightarrow \sigma_V B_1
 \]
 and
 \[
	 \sqrt{n} (F_n - F) \Rightarrow B_2^0 \circ F
 \]
 in $\mathcal{D}[0,\infty)$ as $n \to \infty$; see \cite[Ch. 5]{ChYa2013} and \cite{Du2010}, respectively. Since $F_n$ converges uniformly to $F$ and $B_1$ has continuous sample paths, it follows (using the standard ``random time change" argument; see \cite[pp.144]{Bi1968}) that
 \[
	 \chi_n \circ F_n \Rightarrow \sigma_V B_1 \circ F
 \]
 in $\mathcal{D}[0,T]$ as $n \to \infty$. The independence of $\chi_n$ and $F_n$ then leads to the theorem. 
\qed
 \end{proof}
 
 This result appears implicitly in \cite{HoJaWa2014}, but with a different representation for the limit process than that obtained here. We note that $Z$ is a zero mean Gaussian process with covariance function
 \[
	 \text{Cov}(Z(s), Z(t)) = \sigma^2_V F(s) + (EV_1)^2 F(s)(1-F(t))
 \]
 for $0 \leq s \leq t < \infty$.
\section{The $RS/G/1$ Queue in Heavy Traffic}
Let $W_n(t)$ be the workload in the system at time $t$ in a system with a population of $n$ users, assuming that the system can process $n$ units of work per unit time. If $W_n(0) = 0$, then
\begin{equation} \label{eq:workload}
W_n(t) = \Gamma_n(-\infty,t] - nt - \inf_{0 \leq s \leq t} \left( \Gamma_n(-\infty,s] - ns \right)
\end{equation}
for $t \geq 0$; see \cite{Bo2012}. Proposition \ref{prop:FSLLN} implies the FSLLN
\begin{equation} \label{eq:workload-fslln}
\sup_{t \geq 0} \left| \frac{1}{n} W_n(t) - \sup_{0 \leq s \leq t} \bigg( EV_1(F(t) - F(s)) - (t-s) \bigg) \right| \stackrel{a.s.}{\to} 0
\end{equation}
as $n \to \infty$. It follows from \eqref{eq:workload-fslln} that $W_n$ is of order $n$ at times $t$ for which the FSLLN limit is positive above (with Theorem \ref{thm:basics-fclt} suggesting that $W_n$ is at most of order $\sqrt{n}$ over intervals in which the fluid limit is $0$). One case in which the fluctuations of $W_n$ are of uniform size is the setting in which $EV_1 = 1$ and $F$ is uniform on $[0,1]$. In this special ``balanced" case, Theorem \ref{thm:basics-fclt} immediately implies that
\begin{equation} \label{eq:balanced-fclt}
	\frac{W_n(\cdot)}{\sqrt{n}} \Rightarrow Z(\cdot) - \inf_{0\leq s \leq \cdot} Z(s)
\end{equation}
as $n \to \infty$ in $\mathcal{D}[0,1]$, as a consequence of the continuous mapping principle (and the continuity of the Skorokhod reflection map on $\mathcal{D}[0,1]$; see \cite[Ch. 5]{ChYa2013}).

Of course, in the ``real world,'' exact balance occurs infrequently.  The more general approximation to the workload (for unbalanced systems) would take the form
\begin{equation} \label{eq:unbalanced-approx}
	W_n(t) \stackrel{D}{\approx} \tilde Z_n(t) - ct - \inf_{0 \leq s \leq t} (\tilde Z_n(s) - cs),
\end{equation}
where $\stackrel{D}{\approx}$ denotes ``has approximately the same distribution as...'' (and carries no rigorous justification in general), $c$ is the rate at which the server is able to reduce workload, and $\tilde Z_n$ is a Gaussian process  with mean and covariance that matches $\Gamma_n$.

Our next result formalizes the approximation \eqref{eq:unbalanced-approx} in the ``near balanced case,'' in which the available service effort $c t$ is close to $E \Gamma_n(t)$. To be precise, suppose that in a system that supports $n$ jobs, the service rate obeys
\begin{equation}\label{eq:service-condition}
	c = n E V_1 + a \sqrt{n} + o(\sqrt{n})
\end{equation}
and the scattering distribution $F$ satisfies
\begin{equation}\label{eq:traffic-condition}
	\mathbb{P} (T_1^n \leq t) = t + \frac{q(t)}{\sqrt{n}} + o\left(\frac{1}{\sqrt{n}}\right)
\end{equation}
uniformly on $[0,1]$, where $T_1^n$ is the common scattering random variable associated with system $n$ (for which $\mathbb P(T^n_1 \leq 1) = 1$). Let $\Psi : \mathcal{D}[0,\infty) \to \mathcal{D}[0,\infty)$ represent the functional $\Psi(x)(\cdot) = \sup_{0\leq s \leq \cdot} (x(t) - x(s) )$, and set $e(t) = t$ for $t \geq 0$.

\begin{theorem} \label{thm:workload-unbalanced}
		If  \eqref{eq:service-condition} and \eqref{eq:traffic-condition} are in force and if $EV_1^2 < \infty$, then
		\[
			\frac{W_n(\cdot)}{\sqrt{n}} \Rightarrow \Psi \left(Z+ EV_1 (q - a e \right)(\cdot)
		\]
		as $n \to \infty$ in $\mathcal{D}[0,1]$, and 
		\[
			n^{-1/2} \Psi\left(\tilde Z_n - c e\right) (\cdot) \Rightarrow \Psi \left(Z + EV_1 (q -a e) \right)(\cdot)
		\]
		as $n \to \infty$ in $\mathcal{D}[0,1]$.
\end{theorem}
\begin{proof}
	The key computation here involves
	\[	
		\begin{split}
			n^{-1/2} \left( \Gamma_n(-\infty,t] - c t\right) = \chi_n\left(F_n^{'}(t)\right) &+ EV_1 \sqrt{n} \left(F_n^{'}(t) - \mathbb{P}(T_1^n \leq t)\right)\\ &+ \sqrt{n} \left(EV_1 \mathbb{P}(T_1^n \leq t) - \frac{c}{n}t\right),
		\end{split}
	\]
	where $F_n^{'}$ is the empirical distribution function associate with $T_1^n, \ldots, T_n^n$ (with the $T_i^n$'s being i.i.d. copies of $T_1^n$). According to \eqref{eq:service-condition} and \eqref{eq:traffic-condition},
	\begin{equation} \label{eq:fclt-scale-fluid-flux}
		\sqrt{n} \left(EV_1 \mathbb{P}(T_1^n \leq t) - ct\right) = EV_1 (q(t) - ae(t))) + o(1)
	\end{equation}
	uniformly in $t \in [0,1]$. On the other hand, it is well known that it is possible to find a probability space supporting $\{F_n^{'}, ~n\geq 1\}$ and a sequence of Brownian bridge processes $\{B_n^0, ~n\geq 1\}$ for which
	\begin{equation} \label{eq:empirical-fsat}
		F_n^{'}(\cdot) = \mathbb{P}(T_1^n \leq \cdot) + n^{-1/2} B_n^0\left(\mathbb{P}(T_1^n \leq \cdot)\right) + O\left(\frac{\log n}{n}\right)
	\end{equation}
	uniformly a.s. as $n \to \infty$; see \cite[pp. 132]{CsRe2014}. However, \eqref{eq:traffic-condition} implies that
	\begin{equation}\label{eq:bridge-fsat}
		B_n^0 \left(\mathbb{P}(T_1^n \leq \cdot)\right) = B_n^0 \left(e(\cdot)\right) + O(n^{-p})
	\end{equation}
	uniformly a.s. for $0 < p < 1/4$. Relations \eqref{eq:empirical-fsat} and \eqref{eq:bridge-fsat} together yield the conclusion
	\begin{equation} \label{eq:fclt-empirical}
		\sqrt{n} \left(F_n^{'}(\cdot) - \mathbb{P}(T_1^n \leq \cdot)\right) \Rightarrow B^0(\cdot)
	\end{equation}
	as $n \to \infty$ in $\mathcal{D}[0,1]$. The random time-change theorem (\cite[pp. 114]{Bi1968}) then implies that
	\begin{equation}\label{eq:fclt-counting}
		\chi_n \circ F_n^{'} (\cdot )\Rightarrow \sigma_V B_1 (\cdot)
	\end{equation}
	as $n \to \infty$ in $\mathcal{D}[0,1]$. Results \eqref{eq:fclt-scale-fluid-flux}, \eqref{eq:fclt-empirical} and \eqref{eq:fclt-counting}, together with the continuous mapping principle applied to the Skorokhod map, prove the first part of the theorem.
	
	The second part of the theorem is an immediate consequence of the obvious convergence
	\[
		n^{-1/2} \left(\tilde Z_n (\cdot) - c e(\cdot)\right) \Rightarrow Z(\cdot) + EV_1(q(\cdot) - a)
	\]
	as $n \to \infty$ in $\mathcal{D}[0,\infty)$ (obvious because both sides are Gaussian).
\qed
\end{proof}

%
%
%

We conclude this section with a simple observation that has useful modeling implications. In most real-world applications of a random scattering traffic model, the set of customers that will actually access the resource in a given period is itself random. In particular, not every potential customer will choose to utilize the system in a given period. We can model this by allowing $F$ to have mass at infinity, so that $\lim_{t \to \infty} F(t) < 1$. Theorem \ref{thm:workload-unbalanced} continues to hold, exactly as stated, in this ``sub-probability distribution'' setting, provided that \eqref{eq:traffic-condition} and \eqref{eq:service-condition} are adapted to
\[
	c = n EV_1 ~b + a \sqrt{n} + o(\sqrt{n})
\]
and \[
	\mathbb{P} \left(T_1^n \leq t\right) = bt + \frac{q(t)}{\sqrt{n}} + o \left(\frac{1}{\sqrt{n}}\right)
\]
as $n \to \infty$, where $b=1 - F(\infty)$.

\section{Computing the Distribution of the Reflected Gaussian Process} \label{sec:distribution}
In Section \ref{sec:intro} it was noted that the uniform scattering case is of special interest and has broad applicability in modeling a multitude of traffic phenomena. In this case
\(
	\hat W := \Psi \left(Z - d e\right),
\)
\(
q(t) = 0~\forall t \geq 0,
\)
and $d := a EV_1$.
 Theorem~\ref{thm:proto-dist} below derives the distribution of this process, and highlights the central role the reflected Brownian bridge plays in high intensity regimes - akin to that of the reflected Brownian motion in standard heavy-traffic theory. The computation of the transient distribution of $\hat W$ is complicated when $q(\cdot)$ is non-linear, however. Thus, in Corollary~\ref{cor:proto-dist} we approximate the distribution when the residue function $q$ is ``slowly varying''.

Recall that a Brownian motion process can be expressed in terms of the Brownian bridge as
\[
	B_1(t) = B_1^0(t) + t B_1(1) ~\forall~ 0 \leq t \leq 1,
\]
and it follows that
\(
Z(t) = \sigma_V B_1^0(t) + EV_1 B_2^0(t) + \sigma_V t B_1(1) ~\forall~ 0 \leq t \leq 1.
\)
For fixed $t \in [0,1]$ $\sigma_V B_1^0(t) + EV_1 B_2^0(t) \stackrel{D}{=} \sqrt{EV_1^2} B^0(t)$, where $B^0$ is a Brownian bridge independent of $B_1(1)$. Therefore,
\begin{equation}
	Z(t) \stackrel{D}{=} \sqrt{EV_1^2} B^0(t) + \sigma_V t B_1(1).
\end{equation}
The independence of $B^0$ and $B_1(1)$ implies that $\mathbb{P} \left(\hat W(t)  \leq \lambda \right) = \mathbb P (\sup_{0 \leq s \leq t} (Z(t) - Z(s) - d(t -s)) \leq \lambda)$
\begin{equation} \label{eq:independence-decomposition}
= \int_{\mathbb R} \mathbb{P} \left(\sup_{0 \leq s \leq t} \left(\sqrt{EV_1^2} (B^0_t - B^0_s) - (d- \sigma_V x)(t-s)  \right) \leq \lambda \right) \mathbb{P}(B_1(1) \in dx).
\end{equation}
The main result of this section is a consequence of the fact that the maximum of the Brownian bridge can be expressed in terms of the maximum of the Brownian motion. Further, the joint distribution of the maximum of the Brownian motion and its final position are well known. The distribution of $\hat W$ then follows from these facts, which we present as lemmas (for completeness). 

First, we derive the transient distribution for the reflected Brownian bridge process $Y(t) := \Psi( B^0 - d e)(t)$, where $d \in \mathbb{R}$. This result is a direct consequence of the fact that for the standard Brownian bridge process,
\begin{equation}\label{eq:sup-bb}
\left\{\sup_{0\leq s \leq t} B^0_s \leq \lambda \right\} = \left\{\sup_{0 \leq r \leq u} B(r) - \lambda r \leq \lambda \right\},
\end{equation}
where $B$ is a standard Brownian motion process and $r = s(1-s)^{-1}$ for $0 \leq s \leq t$; see \cite{Do1949,HaWe1980,PeWe2014}. 

\begin{lemma} \label{thm:proto-dist}
	\begin{equation} \label{eq:regulated-bb-dist}
	\mathbb P \left(Y(t) \leq \lambda \right) = \Phi \left(\frac{\lambda (1-2t) + d t}{\sqrt{t(1-t)}}\right) - e^{2 \lambda (\lambda - d)} \Phi \left(\frac{- \lambda + d t}{\sqrt{t (1-t)}}\right),
	\end{equation}
	where  $\Phi$ is the standard normal distribution function.
\end{lemma}

The second lemma is simple and follows from the fact that a weighted sum of normal random variables is normal itself.
\begin{lemma} \label{lem:erf-gauss-integral}
	Suppose $a > 0, b, c \in \mathbb{R}$ then  
	\[
		I(a,b,c) := \int_{-\infty}^{+\infty} e^{-a \xi^2} \Phi(b\xi + c) d\xi = \sqrt{\frac{2 \pi}{a}} \Phi\left(c \sqrt{\frac{a}{a + b^2}}\right).
	\]
\end{lemma}

\begin{theorem}
	The transient distribution of the reflected diffusion process $\hat W(t) = \Psi(Z - d e)(t)$ is $\mathbb{P}\left(\hat W(t) \leq \lambda\right)$
	\begin{equation}
	\begin{split}
	 = \frac{1}{\sqrt{\alpha}} \Phi \left(\gamma \sqrt{\frac{\alpha}{\alpha + \beta^2}}\right)& + \exp\left(2 \lambda^2(1+  \sigma_V^2) - 2 \lambda c\right) \\ \times &\Phi\left(\left(\gamma - 2 \lambda \sqrt{\frac{t}{1-t}} + 2\beta \lambda \sigma_V\right) \left(\sqrt{\frac{2 \alpha^2}{2 \alpha^2 + \beta^2}}\right)\right),
	\end{split}
	\end{equation}
	where $\alpha = 1/2$, $\beta = \frac{-\sigma_V t}{\sqrt{t(1-t)}}$ and $\gamma = \frac{\lambda(1-2t) + ct}{\sqrt{t(1-t)}}$.
\end{theorem}

\begin{proof}
Equations \eqref{eq:independence-decomposition} and \eqref{eq:regulated-bb-dist} together imply that
\[
\begin{split}\mathbb{P} \left(\hat W(t) \leq \lambda\right) = \int_{\mathbb R} &\Phi \left(\frac{\lambda(1-2t) + (c -\sigma_V x)t}{\sqrt{t(1-t)}}\right) \mathbb{P} \left(B_1(1) \in dx\right) \\ 
&- e^{2 \lambda^2} \int_\mathbb{R} e^{-2\lambda(c - \sigma_V x)} \Phi\left(\frac{-\lambda + (c-\sigma_V x)t}{\sqrt{t(1-t)}}\right) \mathbb{P} ( B_1(1) \in dx ). \end{split}
\]

Consider the first integral on the right hand side, and let $\alpha = 1/2$, $\beta = \frac{-\sigma_V t}{\sqrt{t(1-t)}}$ and $\gamma = \frac{\lambda(1-2t) + ct}{\sqrt{t(1-t)}}$, so that
\begin{equation} \label{eq:integral_I1}
	I_1 = \int_{\mathbb{R}} \Phi\left(\beta x + \gamma\right) e^{-\alpha x^2} \mathbb{P}\left(B_1(1) \in dx\right) = \frac{1}{\sqrt{\alpha}} \Phi\left(\gamma\sqrt{\frac{\alpha}{\alpha+\beta^2}}\right),
\end{equation}
by Lemma \ref{lem:erf-gauss-integral}.

For the second integral, ``completing the square'' in the exponent inside the integral we obtain
\begin{equation*}
	I_2 = \frac{e^{-2\lambda c + \lambda^2\sigma^2_V/\alpha}}{\sqrt{2 \pi}} \int_\mathbb{R}  e^{- (\sqrt{\alpha} x - \lambda \sigma_V/\sqrt{\alpha})^2} \Phi\left(\gamma - 2 \lambda \sqrt{\frac{t}{1-t}}+\beta x\right) dx.
\end{equation*}
Changing the variable in the integral to $y = \frac{\alpha x - \lambda \sigma_V}{\sqrt{\alpha}}$, and applying Lemma \ref{lem:erf-gauss-integral} we obtain
\begin{equation}\label{eq:integral_I2}
I_2 = \exp(2 \lambda^2(1+ \sigma_V^2) - 2\lambda c) \Phi \left(\left(\gamma - 2\lambda \sqrt{\frac{t}{1-t}} + \frac{\beta \lambda \sigma_V}{\alpha}\right)\left(\sqrt{\frac{2\alpha^2}{2\alpha^2 + \beta^2}}\right)\right).
\end{equation}
Finally, putting \eqref{eq:integral_I1} and \eqref{eq:integral_I2} together yields the final expression.
\qed
\end{proof}

Next, consider $\hat W (t) = \Psi(B^0 + EV_1 (q - a e))$,  $q$ is defined in~\eqref{eq:traffic-condition}. 
For the next result we assume that $q$ satisfies the following condition:
\begin{assumption}[Slowly Varying] \label{assume:slowly-varying}
  \begin{equation}
    \label{eq:slowly-varying}
    \frac{q(\epsilon e )}{\epsilon} \to q_0 e ~\text{u.o.c. as}~ \epsilon \to 0,
  \end{equation}
  where $q_0 \in \mathbb R$.
\end{assumption}
This ``slow variation'' condition indicates that the first, or linear, variation is determinant of the change in the function $q$ at every time instant $t$. As an example, suppose $q(t) = 1 - e^{-\gamma t}$, then it is straightforward to see that 
\(
\lim_{\epsilon \to 0} \epsilon^{-1} q(\epsilon t) = \gamma t.
\)

Now, consider the scaled stochastic process $X^{\epsilon} :=  \epsilon^{-1} q - \epsilon^{-1} d e + \epsilon^{-1/2} B$. Rescaling time by $\epsilon$ we have
\begin{eqnarray*}
\bar X^{\epsilon} := X^{\epsilon}(\epsilon t) &=& EV_1 \epsilon^{-1} q(\epsilon t) - \epsilon^{-1} d \epsilon t+ \epsilon^{-1/2} B(\epsilon t)\\
&\stackrel{D}{=}& EV_1\epsilon^{-1} q(\epsilon t) - d t + B_t.
\end{eqnarray*}
We claim the following corollary to Theorem~\ref{thm:proto-dist}. The proof is a consequence of the continuous mapping principle applied to the Skorokhod map and Theorem~\ref{thm:proto-dist}.

\begin{corollary}\label{cor:proto-dist}
  Suppose $\tilde q$ satisfies Assumption~\ref{assume:slowly-varying}. Let $\epsilon > 0$, then the diffusion limit workload $\hat W^{\epsilon} = \Psi(X^{\epsilon})$ satisfies
  \begin{equation}
    \label{eq:workload-epsilon-diffusion}
    \hat W^{\epsilon} \Rightarrow \hat W_0 := \Psi \left( Z + E V_1 (q_0 - d) e \right)
  \end{equation}
as $\epsilon \to 0$ in $\mathcal D[0,1]$ and the transient distribution of $\hat W_0$ is $\mathbb P\left(\hat W_0(t) \leq \lambda \right)$
\[
\begin{split}
	 = \frac{1}{\sqrt{\alpha}} \Phi \left(\gamma \sqrt{\frac{\alpha}{\alpha + \beta^2}}\right)& + \exp\left(2 \lambda^2(1+  \sigma_V^2) + 2 \lambda EV_1 (q_0 - d)\right) \\ \times &\Phi\left(\left(\gamma - 2 \lambda \sqrt{\frac{t}{1-t}} + 2\beta \lambda \sigma_V\right) \left(\sqrt{\frac{2 \alpha^2}{2 \alpha^2 + \beta^2}}\right)\right),
	\end{split}
\]
where $\alpha = 1/2$, $\beta = \frac{-\sigma_V t}{\sqrt{t(1-t)}}$ and $\gamma = \frac{\lambda(1-2t) + EV_1 (d - q_0)t}{\sqrt{t(1-t)}}$.
\end{corollary}

Corollary~\ref{cor:proto-dist} shows that when the residue function $q$ is slowly varying over time the distribution of $\hat W$ is well approximated by the distribution of a reflected diffusion with linear drift. We note, however, that deriving the transient distribution for a general reflected Brownian Bridge process with a time-varying drift function is non-trivial and outside the scope of the current paper.

\section{Exact Asymptotics of the Gaussian Approximations}
In this section, we investigate the tail behavior of the Gaussian approximation in Theorem~\ref{thm:workload-unbalanced} by obtaining the exact asymptotics of $\mathbb{P}(\hat W(t) > x)$ as $x \to \infty$. Recall that the covariance function of the Gaussian process $Z$ is for $s < t$
\(
	\text{Cov}(Z(t), Z(s)) = \sigma_V^2 F(s) + (EV_1)^2 F(s)(1-F(t)).
\)
We are mostly interested in the case where the scattering distribution is $F(t) = t$ (or satisfies \eqref{eq:traffic-condition}), where it can be seen that 
\[
\text{Cov}(Z(t), Z(s)) = \sigma_V^2 s + (EV_1)^2 s(1-t) = (EV_1)^2 s \left( c_s^2 + 1 - t \right) ~\forall s,t \in [0,1],
\]
where $c_s^2$ is the coefficient of variation of the service times. To simplify the notation, we consider the scaled random process $\tilde Z := Z/EV_1$. 

The exact asymptotics are a direct consequence of the following result from \cite{PiPr1978}, reproduced below in a convenient form. 

\begin{proposition} \label{prop:piterbarg}
		Let $X$ represent a mean zero Gaussian process with variance time function $EX_t^2 = \sigma_t^2$ and $EX_tX_s = \rho(t,s)$. Suppose that the following three assumptions are fulfilled.

\begin{enumerate}
  \item[$A_1$] $\sigma_t$ has a unique global maximum $t_* \in [0,1]$ and
    \[
       \sigma_t = \sigma - A |t-t_*|^\beta + o\left( |t-t_*|^\beta \right), ~t\to t_*,
    \]
    for some $A > 0$ and $\beta > 0$.
  \item[$A_2$] $\rho(t,s) = 1 - D|t-s|^\alpha + o\left( |t -s|^\alpha\right) $, $t,s \to t_*$, for some $D > 0$ and $\alpha > 0$.
  \item[$A_3$] For some $c > 0$, $C >0$,
    \[
        E(X_t - X_s)^2 \leq C |t-s|^\gamma, ~\forall t,s \in [0,1].
    \]
\end{enumerate}
Under these conditions, if $\beta > \alpha$, then
\begin{equation}
\label{eq:piterbarg}
  \mathbb{P}\left(\max_{t \in [0,1]} X(t) > x\right) \sim \begin{cases} 2 \mathcal{H}(\alpha,\beta) \exp\left(-\frac{x^2}{2 \sigma^2} \right) \left( \frac{x}{\sigma}\right)^{-\left(\frac{2}{\beta} - \frac{2}{\alpha}+1 \right)} &~\text{if}~ t_* \in (0,1)\\ \mathcal{H}(\alpha,\beta) \exp\left(-\frac{x^2}{2 \sigma^2} \right) \left( \frac{x}{\sigma}\right)^{-\left(\frac{2}{\beta} - \frac{2}{\alpha}+1 \right)} &~\text{if}~ t_*=0 \text{ or } t_*=1,\end{cases}
\end{equation}
as $x \to \infty$, and where 
\[
\mathcal{H}(\alpha,\beta) = \frac{\Gamma(1/\beta) D^{1/\alpha} \sigma^{1/\beta} H_\alpha}{\sqrt{2 \pi} \beta A^{1/\beta}},
\]
 
\[
H_\alpha = \lim_{T \to \infty} \frac{1}{T} E \left[\exp\left( \max_{t \in [0,T]} \sqrt{2} B_{\alpha} - t^{2\alpha} \right) \right] < \infty
\]
and $B_\alpha$ is a fractional Brownian motion with hurst index $\alpha \in (0,1]$.
\end{proposition}

\begin{theorem}
\noindent $(i)$ Suppose the squared coefficient of variation of the service times satisfy $c_s^2 < 1$. Then, the exact asymptotics of the Brownian approximation to the workload process follows as 
\begin{equation}
   \mathbb{P}\left(\max_{s\in [0,t]} \hat W(s) > x\right) \sim 2 \mathcal{H}(\alpha,\beta) \exp\left(-\frac{x^2}{2 \sigma^2} \right) \left( \frac{x}{\sigma}\right)^{-\left(\frac{2}{\beta} - \frac{2}{\alpha}+1 \right)}
\end{equation}
as $x \to \infty$, with $\alpha = 1, ~\beta = 2, ~\gamma = 1$ and $D = (c_s^2 + 1)/2$.\\
\noindent $(ii)$ Suppose the squared coefficient of variation of the service times satisfy $c_s^2 \geq 1$. Then, the exact asymptotics of the Brownian approximation to the workload is
\begin{equation}
\mathbb{P}\left(\max_{s\in [0,t]} \hat W(s) > x\right) \sim  \mathcal{H}(\alpha,\beta) \exp\left(-\frac{x^2}{2 \sigma^2} \right) \left( \frac{x}{\sigma}\right)^{-\left(\frac{2}{\beta} - \frac{2}{\alpha}+1 \right)}
\end{equation}
as $x \to \infty$, with $\alpha = 1, ~\beta = 2, ~\gamma = 1$ and $D = (c_s^2 + 1)/2$.
\end{theorem}
\begin{proof}
\noindent (i) Recall that $\hat W(t) = \sup_{0 \leq s \leq t} \left( Z(t) - Z(s) - c(t-s) \right) \stackrel{D}{=} \sup_{0 \leq s \leq t} (Z(s) - cs)$ we have
\[
   \mathbb{P} \left( \hat W(t) > x\right) = \mathbb{P} \left( \sup_{0 \leq s \leq t} Z(s) - c s > x \right).
\]
Define $m_x(t) := \frac{ct + x}{\sigma(t)}$, where $\sigma(t) = \left(t \left(c_s^2 + 1 - t\right) \right)^{1/2}$ to be the variance time curve of the process $Z(t)$. Straightforward algebraic manipulation shows that 
\[
\mathbb{P} \left( \max_{s \in [0,t]} (Z(s) - cs) > x\right) = \mathbb{P} \left( \max_{s \in [0,t]} \frac{Z(s)}{\sigma(s)}\frac{m_x(t_*)}{m_x(s)} > m_x(t_*)\right),
\]
where $t_* = \frac{x(c_s^2+1)}{c(c_s^2 + 1) + 2x}$ is the unique minimizer of $m_x(t)$ for $t \in [0,1]$ ($t_* < 1$ since we have assumed that the service times are low variance). The proof of the theorem follows by verifying that the conditions $A_1$- $A_3$ in Proposition \ref{prop:piterbarg} are satisfied by $\tilde Z(s) := \frac{Z(s)}{\sigma(s)} \frac{m_x(t_*)}{m_x(s)}$ for $s \in [0,t]$. 

The variance time-curve of the process $\tilde Z$ is $\tilde \sigma(t) = \left(\frac{m_x(t_*)}{m_x(t)}\right)^2$. Therefore, by Lemma 7.4.2 in \cite{Ma2007} it follows that condition $A_1$ is satisfied with $\beta = 2$, $\sigma = 1$ and 
\[
	A = -\frac{\sigma''(t_*)}{2\sigma(t_*)} = \frac{1}{4} \frac{x(c_s^2 + 1)^3 \left(c(c_s^2 + 1) + x\right)}{\left(c (c_s^2 + 1) + 2 x\right)^2},
\]
where $\sigma''$ is the second derivative of the variance time curve. Next, the covariance of the scaled process $Z(t)/\sigma(t)$ is 
\[
	\rho(t,s) = \frac{EZ(t) Z(s)}{\sigma(t) \sigma(s)} = \sqrt{\frac{s (c_s^2 + 1 - t)}{t (c_s^2 + 1 -s)}}.
\]
Since $E[(Z(t) - Z(s)^2] = E[Z(t-s)^2]$, it follows that
\[
	1 - \rho(t,s) = \frac{\sigma^2(t-s) - (\sigma(t) - \sigma(s))^2}{2 \sigma(t) \sigma(s)}.
\]
Recall that $\sigma^2(t) = t(c_s^2 + 1 - t)$, so that
\[
	\lim_{t \to 0} \frac{\sigma^2(t)}{t} = c_s^2 + 1. 
\]
Therefore, it follows that
\[
	 1 - \rho(t,s) = \frac{c_s^2 +1}{2} |t-s| (1+o(1)) 
\]
as $t,s \to t_*$. Thus condition $A_2$ is satisfied with $D = \frac{c_s^2 +1}{2}$ and $\alpha = 1$. Finally, it is a simple exercise to see that Lemma 7.4.4 in \cite{Ma2007} implies that condition $A_3$ is satisfied with $\gamma = 1.$ By assumption, $t_* \in (0,1)$, so the first case in~\eqref{eq:piterbarg} in Proposition \ref{prop:piterbarg} implies the first part of the theorem.

\noindent (ii) When $c_s^2 > 1$ the function $m_x(t)$ achieves a unique infimum at $t_* > 1$. Furthermore, the first derivative test shows that $m_x(t)$ is monotone decreasing on the interval $[0,1]$. This implies that the infimum \textit{in $[0,1]$} is achieved at $t = 1$; that is, $t_* = 1$. Then, following the analysis for case (i) above, conditions $A_1$-$A_3$ can be shown to be satisfied, and the theorem is proved as a consequence of the second case in~\eqref{eq:piterbarg} in Proposition~\ref{prop:piterbarg}.
\qed
\end{proof}

\section{Large Workloads in the $RS/G/1$ Queue} \label{sec:large-deviations}
While the previous section was concerned with the extreme-values of the workload diffusion approximation, it is also illuminating to study the rare events of the workload process itself. More precisely, in this section we consider events of the form
\[	
	\{W_n(t) > y\}
\]
when $y > 0$ and $n$ are large (corresponding to a ``high intensity'' environment) and study the logarithmic asymptotics of the probability of these events. In view of \eqref{eq:unbalanced-approx} and Theorem \ref{thm:workload-unbalanced}, a large $y$ implies that
\[
	y - \sup_{0 \leq s \leq t} \left( nEV_1(F(t) - F(s) - c(t-s))\right) >> \sqrt{n},
\]
so that $\{W_n(t) > y\}$ is a rare event (when the service rate $c$ per unit time is of order $n$). Let $\varphi(\theta) = E[e^{\theta V_1}]$ represent the moment generating function of the service time random variable, and consider for each $s \in [0,t]$
\[
	v(s,\theta) = \log \left(\left( \varphi(\theta) - 1\right)(F(t) - F(s)) + 1\right).
\]
Consider the logarithmic moment generating function of the random variable 
\(
	\Gamma_n(0,t] - \Gamma_n(0,s] - c(t-s) :
\)
\[
	\log E \left[\exp\left(\theta \left(\Gamma_n(0,t] - \Gamma_n(0,s] - c(t-s)\right)\right)\right] = n v(s,\theta) - \theta c(t-s).
\]
Let $\theta(s)$ be the root (in $\theta$) of the equation (assuming it exists)
\[
	n \frac{\partial v(s,\theta)}{\partial \theta} - c (t-s) = y
\]
and set 
\[
	I(s) = \theta(s) y  - n v(s,\theta(s)) + \theta(s) c (t-s).
\]
Let $t_{*}$ be the minimizer of $I(\cdot)$ over $[0,t]$. Then, the large deviations approximation we prove is 
\begin{equation}
\mathbb{P} \left(W_n(t) > y \right) \approx e^{-I(t_{*})}.
\end{equation}

We now rigorously justify this in the setting in which $n \to \infty$, $y = nx$, $c = n c'$, with $x$ and $c'$ chosen such that
\begin{equation} \label{eq:x-condition}
	x > \sup_{0 \leq s \leq t} \left(EV_1 \left(F(t) - F(s) \right) - c' (t-s)\right).
\end{equation}
In this scaling regime, it is straightforward to see that $\theta(s)$ satisfies
\begin{equation} \label{eq:theta-critical}
	\frac{\partial v(s,\theta)}{\partial \theta} - c' (t-s) = x
\end{equation}
and $I(s) = n I^{'}(s)$ where
\[
	I^{'}(s) = \theta(s) x - v(s,\theta(s)) + \theta(s) c' (t-s).
\]
To proceed with the analysis, we also assume the following technical condition.
\begin{assumption} \label{assume:sample-density}
	$F$ has a positive continuous density over $[0,t]$.
\end{assumption}

\begin{theorem} \label{thm:workload-ldp}
	Suppose \eqref{eq:theta-critical} has a unique root $\theta(s) \in \mathbb{R}_+$ for each $s \in [0,t]$ and each $x$ satisfying \eqref{eq:x-condition}. Then,
	\[
		\frac{1}{n} \log \mathbb P\left(W_n(t) > nx\right) \to -\inf_{0\leq s \leq t} I^{'}(s)
	\]
	as $n \to \infty$.
\end{theorem}
\begin{proof}
	First note that for each $s \in [0,t]$
	\begin{eqnarray*}
		\mathbb{P} \left(W_n(t) > nx\right) &=& \mathbb{P} \left(\sup_{0 \leq s \leq t} \left(\Gamma_n(0,t] - \Gamma_n(0,s]\right)  - c'(t-s) > nx\right) \\
		&\geq& \mathbb{P} \left(\Gamma_n(s,t]  - c'(t-s) > nx\right),
	\end{eqnarray*}
	so that (in particular),
	\[
		\mathbb{P} \left(W_n(t) > nx\right) \geq \max_{0 \leq s \leq t} \mathbb{P} \left(\Gamma_n(s,t]   - c'(t-s) > nx\right).
	\]
	Since $\log E\left[\exp \left(\theta \Gamma_n(s,t] - \theta nc' (t-s)\right)\right] = n v(s,\theta) - n\theta c' (t-s)$, Cramer's theorem \cite[Ch. 2]{DeZe2009} implies that
	\[
		\lim_{n \to \infty} \frac{1}{n} \log \mathbb{P} \left(\Gamma_n(s,t] - n c'(t-s) > n x\right) = -I^{'}(s),
	\]
	and consequently 
	\begin{equation}
		\liminf_{n \to \infty} \frac{1}{n} \log \mathbb{P} \left(W_n(t) > n x\right) \geq \max_{0 \leq s \leq t} (-I^{'}(s)) = - \min_{0 \leq s \leq t} I^{'}(s).
	\end{equation}
	(Here we set $I^{'}(t) = -\infty$. Assumption \ref{assume:sample-density} implies that $\Gamma_n(s,t] \to 0$ as $s \uparrow t$.)
	
	For the upper bound, note that for $\delta > 0$ and a uniform partition of the interval $[0,t]$, $\mathbb{P} \left(\max_{0\leq s \leq t} \Gamma_n(s,t] - nc' (t-s) > n x\right)$
	\begin{eqnarray*}
		& \begin{split} \leq \mathbb{P} \bigg(  \bigg\{\max_{0 \leq j < m}& \left(\Gamma_n(jt/m,t] - n c'(t - t j/m > n (x - \delta))\right) \bigg\} \\ &\bigcup \left\{ \max_{0 \leq j < m} \max_{0 \leq r \leq 1/m} \left(\Gamma_n(jt/m, jt/m+r] - nc'r \geq n \delta \right) \right\} \bigg) \end{split} \\
		&\begin{split}
		 \leq \sum_{j=0}^{m-1} \mathbb{P} (\Gamma_n(jt/m,t] -nc'(t - jt/m) &> n(x - \delta) )\\ &+ \sum_{i=0}^{m-1} \mathbb{P}\left( \Gamma_n(jt/m,(j+1)t/m] \geq n\delta\right).
		\end{split}
	\end{eqnarray*}
	
	It follows that for each fixed $m \geq 1$ and $\delta > 0$ (by the ``principle of the maximum term''), $\limsup_{n \to \infty} \frac{1}{n} \log \mathbb{P}\left(W_n(t) > nx\right)$
	\[
		 \begin{split} \leq \max_{0\leq j < m} \bigg(  \limsup_{n\to\infty} \frac{1}{n} & \log \mathbb{P}\left(\Gamma_n(jt/m,t] - nc'(t - jt/m) > n(x - \delta)\right),\\&\limsup_{n \to \infty} \frac{1}{n} \log \mathbb{P}\left(\Gamma_n(jt/m, (j+1)t/m] \geq n \delta\right) \bigg).\end{split}
	\]
	But,
	\[
		\Gamma_n(jt/m, (j+1)t/m) \stackrel{st}{\leq} \sum_{i=1}^{n} V_i \mathbf{1}_{\{U_i \leq \| f\|_t /m}\}
	\]
	where $\stackrel{st}{\leq}$ denotes ``stochastically dominated by'' (in the sense of first order stochastic ordering), $U_1,\ldots U_n$ are i.i.d. uniform random variables on $[0,1]$, and $\|f\|_t := \sup_{0 \leq s \leq t} |f(s)|$. Hence, $	\limsup_{n \to \infty} \frac{1}{n} \log \mathbb{P}\left(W_n(t) > nx\right)$
	\begin{equation} \label{eq:log-prob-upper}
	 \begin{split} \leq \max_{0\leq j < m} \bigg( \limsup_{n\to\infty} \frac{1}{n} &\log \mathbb{P}\left(\Gamma_n(jt/m,t] - nc'(t - jt/m) > n(x - \delta)\right),\\& \limsup_{n \to \infty} \frac{1}{n} \log \mathbb{P}\left( \sum_{i=1}^{n} V_i \mathbf{1}_{\{U_i \leq \| f\|_t/m}\} \geq n \delta\right) \bigg). \end{split}
	\end{equation}
	
	It is straightforward to see that 
	\[
		\limsup_{n \to \infty} \frac{1}{n} \log \mathbb{P} \left(\sum_{i=1}^{n} V_i \mathbf{1}_{\{U_i \leq \| f\|_t/m}\} \geq n \delta\right) \to - \infty
	\]
	as $m \to \infty$, so it follows that by sending $m \to \infty$ in \eqref{eq:log-prob-upper} we obtain
	\[
		\limsup_{n \to \infty} \frac{1}{n} \log \mathbb{P}\left(W_n(t) > nx\right) \leq \sup_{0 \leq s \leq t} \limsup_{n \to \infty} \frac{1}{n} \log \mathbb{P}\left(\Gamma_n(s,t] - n c' (t-s) > n (x-\delta)\right).
	\]
	Finally, we send $\delta \downarrow 0$, yielding the upper bound
	\[
		\limsup_{n \to \infty} \frac{1}{n} \log \mathbb{P} \left(W_n(t) > nx\right) \leq \sup_{0 \leq s \leq t} -I^{'}(s)
	\]
	and proving the theorem.
\qed
\end{proof}

We now identify the ``rare event'' arrival paths associated with the large deviations computation of Theorem \ref{thm:workload-ldp}. We assume that $I^{'}(s)$ is minimized at a unique $t_*$ in $[0,t]$. Consider the `twisted' measure
\[
	\mathbb P_n^*(\cdot) = E\left[e^{(\theta(t_*) \Gamma_n(t_*,t)} e^{-n v(t_*,\theta(t_*))} \mathbf 1_{\{ \cdot\}} \right].
\] 
Notice that the pairs $(V_i,T_i)$ are i.i.d. under this measure. However, $V_i$ is not independent of $T_i$. In particular, it can be seen that
\begin{equation}
  \label{eq:1}
  \mathbb P_n^{*}(T_i \in ds) = \begin{cases} \mathbb{P} (T_i \in ds), ~ &s \not \in (t^*,t] \\
    \mathbb P(T_i \in ds) E\left[ e^{\theta(t^*) V_1 - v(t^*,\theta(t^*))}\right], ~&s \in (t^*,t],
\end{cases}
\end{equation}
and
\begin{equation}
  \label{eq:2}
  \mathbb P_n^*(V_i \in dv | T_i = s) = \begin{cases}
    \mathbb P(V_i \in dv), ~ &s \not \in (t^*,t]\\
    \frac{e^{\theta(t^*)v}}{E[e^{\theta(t^*) V_i}]} \mathbb P(V_i \in dv) , ~ &s \in (t^*,t].
\end{cases}
\end{equation}
Furthermore, under $P_n^*$, $\Gamma_n(-\infty,s]$ has expectation
\begin{equation}
  \label{eq:3}
  \bar \Gamma_n^*(s) := E_{\mathbb P_n^*}\left[\Gamma_n(0,s]\right] = \begin{cases}
    E V_1 F(s) e^{-v(t^*,\theta(t^*))}, ~ &s \leq t^*,\\ \\
    \begin{split}\bigg( E V_1& F(t^*) + E \left[V_1e^{\theta(t^*) V_1} \right]\times\\& (F(t) - F(s))\bigg) e^{-v(t^*,\theta(t^*))}, \end{split} ~ &s \in (t^*,t].
\end{cases}
\end{equation}
This expectation represents the expected number of arrivals under the twisted distribution, at any time $s \leq t$, or the ``rare event path'' of the arrivals that most likely ensures that $\{W^n(t) > y\}$. We illustrate these results with a couple of examples.

\begin{example}

\begin{figure}[th] \label{fig:uniexp}
  \centering
  \includegraphics[scale=0.25]{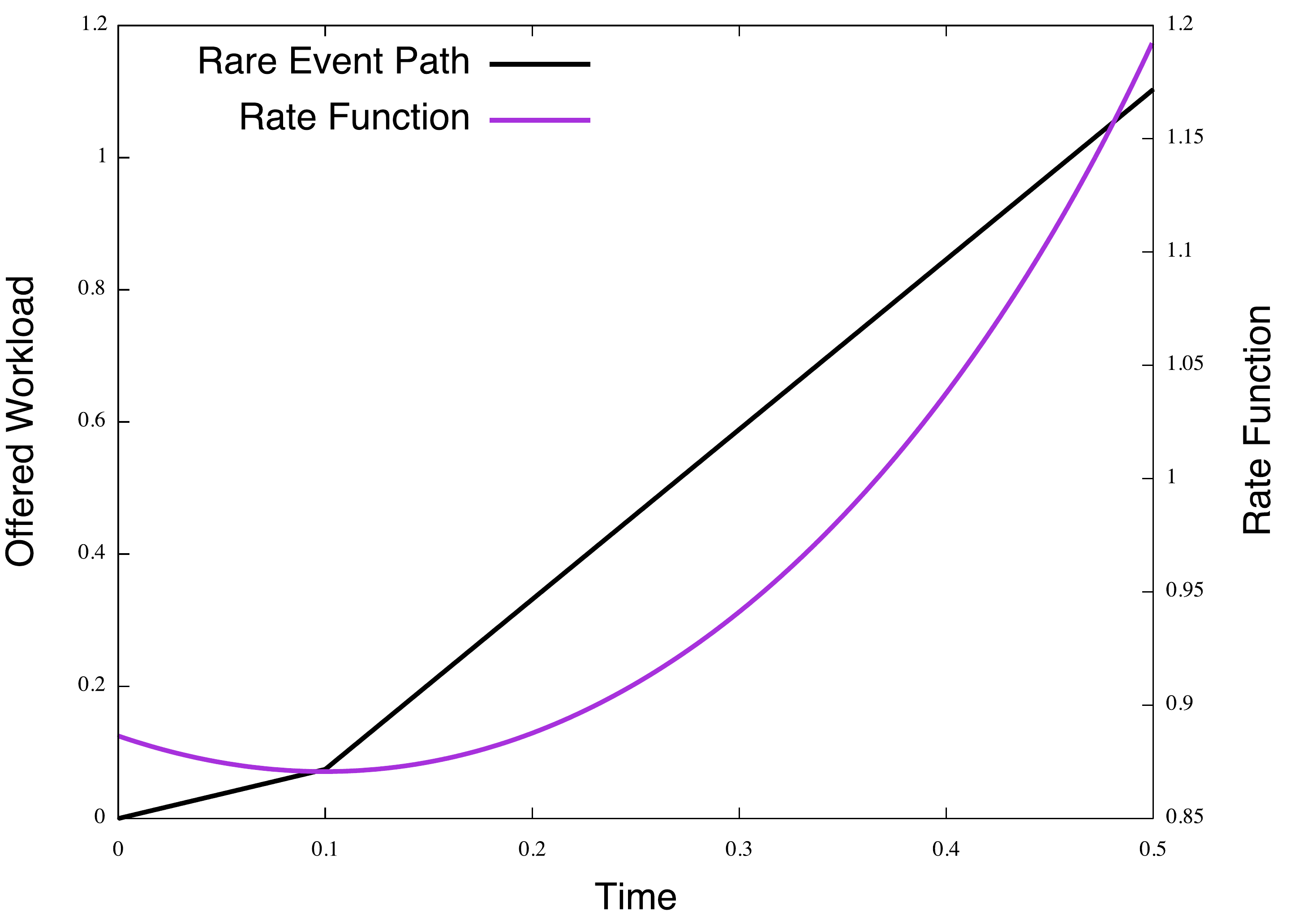}
  \caption{Rate function and rare event path for a queue with unit mean exponential service times and uniformly distributed arrivals in $[0,1]$, when $t = 0.5, ~c' = 1.03$ and $x = 0.5$.}
\end{figure}

  Let $V_1$ be exponentially distributed with mean $1$ and the arrival epochs uniformly distributed over $[0,1]$. It follows that 
  \begin{equation*}
    v(s,\theta) = \log \left( \frac{\theta}{1-\theta}(t-s) + 1\right), ~ s,t \in [0,1]
  \end{equation*}
provided, of course, $\theta < 1$. Recalling the definition of $c$ in \eqref{eq:service-condition} and ~\eqref{eq:x-condition} it follows that we want to carefully choose $x$ and $c'$ such that $x > t(1-c')$. Some algebra shows that the root of equation \eqref{eq:theta-critical} is $\theta(s) = $
\[
\frac{1}{2a(x+c'a)} \left[ (1+a)(x + c' a) - ((1+a)^2(x + c'a)^2 - 4 a(x+c'a)(x+a(c'-1)))^{1/2}\right],
\]
where $a := t-s$ (for brevity). It can be verfied that $\theta(s) < 1$. If $x > t(1-c')$, then $I'(s)$ is a convex function with a unique minimizer $t^*$ in the interval $[0,t]$. In which case, the rare event path is
\begin{equation*}
  \bar \Gamma_n^*(s) = \begin{cases}
    s\frac{1-\theta(t^*)}{1-\theta(t^*)(1-(t-s))}, ~ &s \leq t^*\\
\left(t^* + \frac{(t-s)}{(1 -\theta(t^*))^2}\right) \frac{1-\theta(t^*)}{1-\theta(t^*)(1-(t-s))}, ~ &s \in (t^*,t].
\end{cases}
\end{equation*}

As a specific example, suppose that $t = 0.5$, $c' = 1.03$ and $x = 0.5$, implying  $t^* = 0.1$. Figure \ref{fig:uniexp} depicts the rate function and the most likely rare event path of the offered work to the system. The offered workload builds at a low rate up to $t^*$ and then starts to grow much faster in $(t^*,t]$. Thus, the bulk of the workload enters in the latter interval.
\end{example}

\begin{example}
  \begin{figure}[ht]\label{fig:expexp}
    \centering
    \includegraphics[scale=0.25]{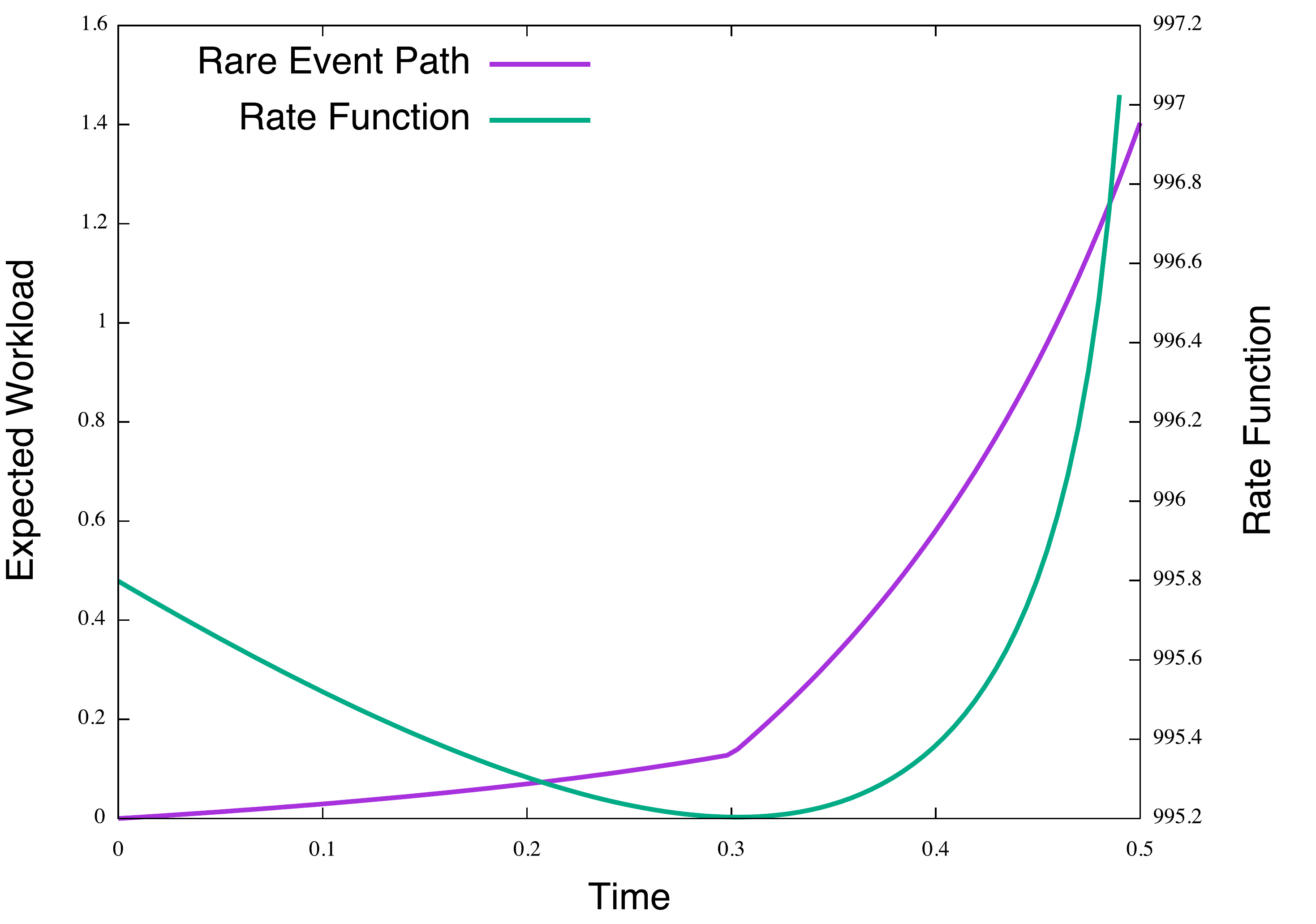}
    \caption{Rate function and rare event path for a queue with unit mean exponential service times and exponentially distributed arrivals (with parameter $\lambda = 1$) epochs in $[0,\infty)$, when $t = 0.5, ~ c' = 5.6$ and $x = 1000.0$.}
  \end{figure}

Now, suppose that arrival epochs are randomly scattered per an exponential distribution with parameter $\lambda = 1$. We continue to assume that the service times are exponentially distributed with mean $1$. It is important to note that this is \textit{not} an $M/M/1$ queue, since we are not modeling the inter-arrival times as exponentially distributed random variables. Running through the computations again, it can be seen that $2 \theta(s) (e^t+e^s - 1) (x+c'a)=$
\begin{equation*}
\begin{split} ( 2ac' e^t + a c' e^a & - a c' + x e^a+2 x e^t - x -\\ &\sqrt((e^a-1)(x+c'a)(a c' e^a - ac + x e^a - 4e^a - 4e^t - x + 4)) ). \end{split} 
\end{equation*}
Recall that we require $x > \sup_{0\leq s \leq t} \left( e^{-s} - e^{-t} - c'(t-s) \right)$. Thus, the choice of $c'$ and $x$ depends on $t$. For illustration, suppose $t = 0.5$, then some experimentation shows that an appropriate choice of parameters is $c' = 5.6$ and $x = 1000.0$. In this case, the rate function is convex and has a unique minimizer at $t^* = 0.3$. The rare event path can be computed by substituting $\theta(t^*)$ into \eqref{eq:3}, along with $F(t) = 1 - e^{-t}$. Figure \ref{fig:expexp} depicts the rate function and the rare event path. The rare event path, once again, is piecewise but non-linear. This is unlike the behavior of the standard $G/G/1$ queue (see \cite{An1989}, for instance).
\end{example}

\section{Periodic High Intensity Traffic}
Now consider a model of traffic where arrivals are scattered periodically according to a common scattering distribution. To keep the description simple, we assume that the population of users is $n$ in each time slot (period) of length $T$ time units, an obvious and natural extension of the model in Section \ref{sec:gauss-limits}. 
Our results significantly extend the analysis in \cite{Ha1994}, where the stationary queue length distribution of a single server FIFO queue with constant service rate and periodically and uniformly scattered traffic was studied, by considering arbitrary scattering distributions and more general service models, and by focusing on the workload process.

Consider the periodic offered load process,
\begin{equation}
  \label{eq:9}
  \Gamma_n(-\infty,t] := \sum_{l=1}^{p_t - 1} \sum_{i=1}^n V_i^{(l)} \mathbf{1}_{\{T_i \in ((l-1)T,lT)]\}} + \sum_{i=1}^{n} V_i^{(p_t)} \mathbf 1_{\{T_i \in ((p_t-1)T,t]\}},
\end{equation}
where $p_t := \inf \{m \in \mathbb N : t \leq m T\}$ is the time slot corresponding to time $t$, and $V_i^{(l)}$ is the service time requested by the $i$th job in time slot $l$. If $\Gamma_n^{(m)}(t) := \sum_{i=1}^{n} V_i^{(m)} \mathbf 1_{\{T_i \in ((m-1)T,t]\}}$ is the offered workload in slot $m$, it is straightforward to see that
\begin{equation}
  \label{eq:4}
  \Gamma_n(-\infty,t] \stackrel{D}{=} \sum_{l=1}^{p_t - 1} \sum_{i=1}^n V_i^{(l)} + \Gamma_n^{(p_t)}(t).
\end{equation}

Consider $t \in [0,T$ and observe $\Gamma_n(t+T) = \Gamma_n(t) + \sum_{i=1}^n V_i^{(1)}$, implying that
\begin{equation}
  \label{eq:5}
  \text{Cov} \left(\Gamma_n(-\infty,t],\Gamma_n(s,t+T]\right) = \begin{cases}
    0 ~& \text{if}~ s \in (T,t+T]\\
    \text{Cov} \left( \Gamma_n(-\infty,t],\Gamma_n(s,T]\right) ~&\text{if}~ s \in (t,T].
\end{cases}
\end{equation}
Following the commentary in Section \ref{sec:gauss-limits}, in the second case above, the correlation is negative. Thus, the process is negatively correlated within a time-slot - which is unsurprising, given our assumptions on the model - and, at longer time-scales, the offered load is uncorrelated. Notice that a similar conclusion will hold if we consider the correlation between $\Gamma_n(-\infty,t]$ and $\Gamma_n(s,t+ mT]$ for any $m \geq 1$.

 Now, let $F_m(t) := F(t - (m-1)T)$ represent the scattering distribution in slot $m$, $\{B_{m}, m \in \mathbb N\}$ a sequence of independent Gaussian processes where
 \begin{equation} \label{eq:extended-BM}
     B_m(t) \stackrel{D}{=} 
   \begin{cases}
     0 ~&\text{if}~ t \not \in ((m-1)T,mT]\\
     B(t-(m-1)T) ~&\text{if}~ t  \in ((m-1)T,mT]
   \end{cases}
 \end{equation}
and $B$ is a standard Brownian motion, and $\{B^0_m, m\in \mathbb{N}\}$ a sequence of standard Brownian bridge processes, with $B^0_m$ defined analogously to \eqref{eq:extended-BM}. Mirroring Proposition \ref{prop:FSLLN} and Theorem \ref{thm:basics-fclt} we have the following result.

\begin{theorem} \label{prop:FSLLN-period}
  \noindent (i) If $EV_1 < \infty$ then,
  \begin{equation}
    \label{eq:6}
    \sup_{0 \leq s < t < \infty} \left| \frac{1}{n} \Gamma_n(s,t] - (p_t - p_s) EV_1 - EV_1 (F_{p_t}(t) - F_{p_s}(s)) \right|    \stackrel{a.s.}{\to} 0
  \end{equation}
as $n \to \infty$.\\
\noindent (ii) If $EV_1^2 < \infty$, then
\begin{equation}
  \label{eq:7}
  \sqrt{n} \left( \Gamma_n(-\infty,\cdot] - EV_1(p_\cdot - 1 + F_{p_\cdot}(\cdot)) \right) \Rightarrow \tilde Z(\cdot)
\end{equation}
as $n \to \infty$, where $\tilde Z(t) := \mathcal{N}(0,\sqrt{p_{t} - 1} \sigma_V) + \sigma_V B_{p_t} \circ F_{p_t} + EV_1 B^0_{p_t} \circ F_{p_t}(t)$ and $\mathcal{N}(0,1)$ is a standard normal random variable.
\end{theorem}

Note that the random variable $\mathcal{N}(0,\sqrt{p_{t} - 1} \sigma_V) \stackrel{D}{=} \sum_{l=0}^{p_t-1} B_l(lT)$ arises from the cumulative service requested up to time slot $p_t - 1$; clearly, over a single period, $\tilde Z = Z$ as defined in equation \eqref{eq:single-period-diffusion}. The proof of this theorem follows from standard arguments and is omitted. In the remainder of this section we assume that 
$T=1$ without loss of generality.

Our goal in this section is to characterize the steady state distribution of the workload process in the high intensity limit. Clearly, the distribution of the service requirements plays a crucial role, and we will characterize the steady state distribution under three different conditions on the service requirements and service rate:

\begin{enumerate}
\item Deterministic service requirements and unit service rate,
\item Random, but bounded, service requirements and service rate that satisfies condition \eqref{eq:service-condition}, and
\item Random service requirements with finite second moments and service rate that satisfies condition \eqref{eq:service-condition}.
\end{enumerate}

Suppose the service requirements are identical and deterministic with $V_1 = 1$, and that the server offers a unit service rate. This case was studied in \cite{Ha1994} where, by simple sample path arguments, it was shown that both the pre-limit and diffusion limit workload processes are periodic, so that
\begin{equation}
  \label{eq:8}
  \frac{W_n(t)}{\sqrt{n}} = \frac{1}{\sqrt{n}}\sup_{0 \leq s \leq 1} \left( \Gamma_n(t) - \Gamma_n(t-s) - s\right) \Rightarrow \tilde W(t) := \sup_{0 \leq s \leq 1} \left(  Z(t) - Z(t-s) \right),
\end{equation}
where $Z$ is a periodic extension of the standard Brownian bridge process. Furthermore, the stationary increments of $\Gamma_n$ imply $Z$ is stationary as well. Thus, the stationary distribution of $\tilde W$ is equal to the distribution of $\sup_{0 \leq  s \leq 1} (B^0(s) - s)$, where $B^0$ is a standard Brownian Bridge process. Following Doob in \cite{Do1949} the steady state distribution of $\tilde W$ is
\begin{equation}~\label{eq:periodic-det-service}
\mathbb P(\tilde W(t) \leq x) = 1 - \exp(-2 x (1+x)).
\end{equation}


Next, suppose the service requirements $V_1$ are random with mean $EV_1$ and bounded such that $\Gamma_n(t) - \Gamma_n(t-1) \leq c$. We also assume that the service rate $c$ satisfies condition \eqref{eq:service-condition}. Then, the condition on the offered load implies that 
\[
W_n(t) = \sup_{0 \leq s \leq t} (\Gamma_n(t) - \Gamma_n(s) - c (t-s)) = \sup_{0 \leq s \leq 1} \left( \Gamma_n(t) - \Gamma_n(t-s) - cs \right).
\]
Consequently, $W_n(t)$ has period 1 sample paths, and therefore 
\[
W_n(t) \stackrel{D}{=} \sup_{0 \leq s \leq 1} (\Gamma_n(s) - cs) ~\text{for all}~t \geq 1.
\]
 
Now, the proof of Theorem~\ref{thm:workload-unbalanced} part (ii) can be used to establish the FCLT in this periodic case in a straightforward manner. Indeed, it can be shown that the workload diffusion approximation is $\tilde W := \hat Y - a e - \inf_{0 \leq s \leq \cdot} (\hat Y(s) - a s)$, where $\hat Y$ is the periodic extension of $Z := \sigma_V B_1 + EV_1 B^0_2$ (i.e., equation \eqref{eq:single-period-diffusion} with $F(t) = t ~\forall~ t \in [0,1]$). Our first result establishes the steady state distribution of $\tilde W$:

\begin{proposition} \label{lem:periodic-extension}
  The steady state distribution of $\tilde W$ is
  \begin{equation}
    \label{eq:12}
    \mathbb P\left( \tilde W(t) \leq \lambda\right) = \varphi(\lambda,a) := 1 - \exp \left( -\frac{2 \lambda (\lambda + a)}{EV_1^2} + 2 \frac{\lambda^2 \sigma_V^2}{(EV_1^2)^2}\right) ~t \in [1,\infty).
  \end{equation}
\end{proposition}
\begin{proof}
  The periodcity of the sample paths implies that
  \begin{eqnarray*}
    \tilde W(t) &=& \sup_{0\leq s \leq t} \left( \hat Y(t) - \hat Y(s) - a (t-s) \right)\\
    &=& \sup_{t-1 \leq s \leq t} \left( \hat Y(t) - \hat Y(s) - a (t-s)\right)\\
    &=& \sup_{0 \leq t \leq 1} \left( \hat Y(t) - \hat Y(t-u) - a u \right).
  \end{eqnarray*}
However, $\sup_{0 \leq t \leq 1} \left( \hat Y(t) - \hat Y(t-u) - a u \right)$ is equal in distribution to $\sup_{0 \leq u \leq 1} (\hat Y(u) - a u)$. It follows that
  \begin{eqnarray*}
    \mathbb P( \tilde W(t) \leq \lambda) &=& \mathbb P \left( \sup_{0 \leq u \leq 1} (\sigma_V B_1 + \sqrt{EV_1^2} B^0(u) - a u) \leq \lambda \right)\\
    &=& \int_{\mathbb{R}}  \mathbb P \left( \sup_{0 \leq u \leq 1} (\sigma_V x + \sqrt{EV_1^2} B^0(u) - a u) \leq \lambda \right) \mathbb P(B_1 \in dx)\\
    &=& \int_{\mathbb R} \mathbb P \left( \sup_{0 \leq u \leq 1} \left( B^0(u) - \frac{au}{\sqrt{EV_1^2}} \right) \leq \frac{\lambda - \sigma_V x}{\sqrt{EV_1^2}}\right) \mathbb P(B(1) \in dx).
  \end{eqnarray*}
The final expression follows as a consequence of Doobs' result \cite{Do1949}, which states that
\[
\mathbb P (\sup_{0 \leq u \leq 1} (B^0(u) - a u) \leq \lambda) = 1 - \exp(-2\lambda(a + \lambda)).
\]
\qed
\end{proof}
Note that this result establishes the periodic quasi-steady state distribution at any time $t + kT$, for $t \in [0,1]]$ and $k \in \mathbb N$. 

For general service requirements, however, the sample path periodicity of $\tilde W_n$ no longer holds, and we need to work a little harder. Applying the proof of Theorem \ref{thm:workload-unbalanced} part (ii), it can be seen that the diffusion approximation to $W_n/\sqrt{n}$ is $\hat W(t) = \sup_{0 \leq s \leq t} (\tilde Z(t) - \tilde Z(s) - a (t-s))$. Notice that equation \eqref{eq:extended-BM} (and the equivalent extension for the Brownian bridge) implies the process $\tilde Z$ is equal in distribution to
\begin{eqnarray}
  \label{eq:9}
  \tilde Z(t) &\stackrel{D}{=}& (p_t-1)\sigma_V B(1) + \sigma_V B(t-(p_t-1)) + EV_1 B^0(t-(p_t-1)),\\
\nonumber
&=:& (p_t-1)\sigma_V B(1)+ \hat Y(t-(p_t-1))
\end{eqnarray}
when $F_{p_\cdot}$ is uniform and $T=1$. Observe that (for a fixed sample path) the sequence $\{\tilde Z(p_t) = \sigma_V p_t B(1),~t \geq 1\}$ is monotone, modulo the sign of $B(1)$. This monotonicity implies that the typical sample path of $\tilde W$ is not periodic. However, under the assumption that the variance of the service requirement satisfies $\sigma_V^2 < 1$, a steady state distribution exists. Our next result establishes this result.

\begin{theorem} \label{thm:periodic-extension}
  Let $t \in [1,\infty)$, then
  \begin{equation}
    \label{eq:13}
    \mathbb P \left( \hat W(t) \leq \lambda \right) = \varphi(\lambda,a) \Phi(a) + \frac{1}{\sqrt{2 \pi}}\int_a^{\infty} \varphi(\lambda - \sigma_V(p_t-1)x,a) e^{-x^2/2} dx.
  \end{equation}
  Further, if $\sigma_V^2 < 1$, the steady state distribution of the workload process is
  \begin{equation}
    \label{eq:14}
    \mathbb P \left( \hat{W}(t) \leq \lambda \right) = 1 - \Phi(a) \exp \left( -\frac{2 \lambda (\lambda + a)}{EV_1^2} + 2 \frac{\lambda^2 \sigma_V^2}{(EV_1^2)^2} \right).
  \end{equation}
\end{theorem}
\begin{proof}
  First, suppose that $B(1) > a$. In this case $\{Z(p_t), ~p_t = 1,2,\ldots\}$ is monotonically increasing, implying that
  \begin{equation*}
    \hat W(t) \stackrel{D}{=} (\hat Y(t) + \sigma_V(p_t - 1) B(1) - a t) - \inf_{0 \leq s \leq 1} (\hat Y(s) - a s).
  \end{equation*}
Using the definition of $\tilde W(t)$ it follows that
\begin{eqnarray*}
\hat W(t) &\stackrel{D}{=}& \sup_{t-1 \leq s \leq t} (\hat Y(t) - \hat Y(s) - a (t-s)) + \sigma_V (p_t -1) B(1)\\
&=& \tilde W(t) + \sigma_V (p_t - 1) B(1).
\end{eqnarray*}
Therefore, it follows that
\begin{equation*}
  \mathbb P \left( \hat W(t) \leq \lambda | B(1) > a\right) = \mathbb P \left( \tilde W(t) \leq \lambda - \sigma_V (p_t -1) B(1) | B(1) > a \right).
\end{equation*}

Next, if $B(1) \leq a$ then $\{Z(p_t)\}$ is monotonically decreasing. Consequently, 
\[
\hat W(t) \stackrel{D}{=} \hat Y(t) + \sigma_V (p_t - 1) B(1)- at - \inf_{t-1 \leq s \leq t} (\hat Y(s) + \sigma_V(p_t -1)B(1) - as).
\]
By definition,
\[
\tilde W(t) = \hat Z(t) - \sigma_V (p_t -1) B(1) - at - \inf_{t-1 \leq s \leq t} ( \hat Z(s) - as - \sigma_V(p_s - 1) B(1)).
\]
Now, if $s \in [p_t-1,t]$, then $p_s = p_t$. On the other hand, if $s \in [t-1,p_t-1)$, then $p_s = p_t-1$. As a result
\[
\hat W(t) \stackrel{D}{=} \hat Z(t) - at - \inf_{p_t-1 \leq s \leq p_t} (\hat Z(s) - a s) = \tilde W(t).
\]

The distribution of $\hat W(t)$ follows after conditioning on the state of $B(1)$ and using Proposition \ref{lem:periodic-extension}:
\begin{eqnarray*}
 \begin{split} \mathbb P \left( \hat W(t) \leq \lambda\right) = \int_{-\infty}^a &\mathbb P (\hat W(t) \leq \lambda | B(1) = x) \mathbb P(B(1) \in dx)\\ &+  \int_{a}^{+\infty} \mathbb P (\hat W(t) \leq \lambda | B(1) = x) \mathbb P(B(1) \in dx) \end{split}\\
= \varphi(\lambda,a) \Phi(a) + \frac{1}{\sqrt{2 \pi}}\int_a^{\infty} \varphi(\lambda - \sigma_V(p_t-1)x,a) e^{-x^2/2} dx,
\end{eqnarray*}
where $\varphi(\lambda,a)$ is defined in~\eqref{eq:12}.

Recall that, by definition, $p_t \to \infty$ as $t \to \infty$, and 
\begin{equation*}
  \varphi(\lambda - \sigma_V(p_t-1) x,a) = 1 - \exp\left(2 (\sigma_V^2-1)(\lambda - \sigma_V x (p_t-1))^2\right) \exp \left(-2a\lambda \right) \exp \left( 2a \sigma_V x (p_t-1) \right),
\end{equation*}
where we have assumed (without loss of generality) that $EV_1^2 = 1$. Since $(EV_1)^2 = 1 - \sigma_V^2 > 0$
\[
\lim_{t \to \infty} \exp\left(2 (\sigma_V^2-1)(\lambda - \sigma_V x (p_t-1))^2\right) = 0,
\]
and 
\[
\lim_{t \to \infty} \exp \left( 2a \sigma_V x (p_t-1) \right) = +\infty,
\]
since $x \geq a > 0$, implying 
\begin{equation}
  \label{eq:11}
  \lim_{t \to \infty}  \varphi(\lambda - \sigma_V(p_t-1) x,a) = 1.
\end{equation}
Observe that $ \varphi(\lambda - \sigma_V(p_t-1) x, a) \leq 1 - e^{-2 a \lambda}$ for all $t \geq 1$. Therefore, by the bounded convergence theorem,
\[
\lim_{t \to \infty} \mathbb P(\hat W(t) \leq \lambda) = \varphi(\lambda,a) \Phi(a) + (1 - \Phi(a)),
\]
and the final expression follows by substituting for $\varphi(\lambda,a)$.
\qed
\end{proof}
Some commentary on this theorem is warranted. First, this result establishes the periodic quasi-steady state distribution for any time $t + k$, where $t \in [0,1]$ and $k \in \mathbb N$ when the variance of the service requirements is strictly bounded above by $1$. Second, the variance bound is the natural generalization of the strict bound on the service requirement in Proposition \ref{lem:periodic-extension}. The variance bound is not a significant restriction on the result, however, since the former can be satisfied in practice by rescaling the service times by a sufficiently large constant - or equivalently by accelerating the service rate $c$. 

It is also interesting to observe that the steady state distributions in both Proposition~\ref{lem:periodic-extension} and Theorem~\ref{thm:periodic-extension}, where the service times are stochastic, and the steady state distribution in~\eqref{eq:periodic-det-service} and \cite{Ha1994} where the service times are deterministic. In the latter case, the steady state distribution is negative exponential, akin to a $G/G/1$ queue. In contrast, in the former cases, the steady state distribution is not negative exponential but a translated Weibull distribution, which is rather atypical of such queues. This connection with extreme-value distributions has not been reported before to the best of our knowledge.
\bibliography{refs}

\begin{thebibliography}{10}

\bibitem{HoJaWa2014}
Harsha Honnappa, Rahul Jain, and Amy~R Ward.
\newblock A queueing model with independent arrivals, and its fluid and
  diffusion limits.
\newblock {\em Queueing Systems}, 80(1-2):71--103, 2014.

\bibitem{An1989}
Venkat Anantharam.
\newblock How large delays build up in a {GI/G/1} queue.
\newblock {\em Queueing Systems}, 5(4):345--367, 1989.

\bibitem{GaLePe1975}
Donald~P Gaver, John~P Lehorsky, and Manuel Perlas.
\newblock Service systems with transitory demand.
\newblock In {\em Logistics}, volume~1. 1975.

\bibitem{Ne1982}
Gordon~F. Newell.
\newblock {\em Applications of Queueing Theory}.
\newblock Chapman and Hall, Ltd., London, 1982.

\bibitem{Lo1994}
Guy Louchard.
\newblock Large finite population queueing systems. the single-server model.
\newblock {\em Stochastic Processes and their Applications}, 53(1):117--145,
  1994.

\bibitem{JaJuSh2011}
R.~Jain, S.~Juneja, and N.~Shimkin.
\newblock The concert queueing game: To wait or to be late.
\newblock {\em Discrete Event Dyn. Syst.}, 21(1):103--134, 2011.

\bibitem{HoJa2015}
Harsha Honnappa and Rahul Jain.
\newblock Strategic arrivals into queueing networks: The network concert
  queueing game.
\newblock {\em Operations Research}, 63(1):247--259, 2015.

\bibitem{Du2010}
Rick Durrett.
\newblock {\em Probability: Theory and Examples}.
\newblock Cambridge University Press, 2010.

\bibitem{BoSa2012}
Andrei~N Borodin and Paavo Salminen.
\newblock {\em Handbook of Brownian Motion - Facts and Formulae}.
\newblock Birkh{\"a}user, 2012.

\bibitem{ChYa2013}
Hong Chen and David~D Yao.
\newblock {\em Fundamentals of Queueing Networks: Performance, Asymptotics, and
  Optimization}, volume~46.
\newblock Springer Science and Business Media, 2013.

\bibitem{Bi1968}
Patrick Billingsley.
\newblock {\em Convergence of Probability Measures}.
\newblock John Wiley and Sons, 1968.

\bibitem{Bo2012}
Alexander~A Borovkov.
\newblock {\em Stochastic Processes in Queueing Theory}, volume~4.
\newblock Springer Science and Business Media, 2012.

\bibitem{CsRe2014}
Miklos Cs{\"o}rgo and P{\'a}l R{\'e}v{\'e}sz.
\newblock {\em Strong Approximations in Probability and Statistics}.
\newblock Academic Press, 2014.

\bibitem{Do1949}
Joseph~L Doob et~al.
\newblock Heuristic approach to the kolmogorov-smirnov theorems.
\newblock {\em The Annals of Mathematical Statistics}, 20(3):393--403, 1949.

\bibitem{HaWe1980}
Wendy~J Hall and Jon~A Wellner.
\newblock Confidence bands for a survival curve from censored data.
\newblock {\em Biometrika}, 67(1):133--143, 1980.

\bibitem{PeWe2014}
Mihael Perman and Jon~A Wellner.
\newblock An excursion approach to maxima of the brownian bridge.
\newblock {\em Stochastic Processes and their Applications}, 124(9):3106--3120,
  2014.

\bibitem{DeZe2009}
Amir Dembo and Ofer Zeitouni.
\newblock {\em Large Deviations Techniques and Applications}, volume~38.
\newblock Springer Science and Business Media, 2009.

\bibitem{PiPr1978}
V~I Piterbarg and V~P Prisyajnuk.
\newblock High excursion probability for a gaussian nonstationary process.
\newblock {\em Probability Theory and Mathematical Statistics (in Russian)},
  23(1):185--189, 1978.

\bibitem{Ma2007}
Michel Mandjes.
\newblock {\em Large deviations for Gaussian Queues: Modelling Communication
  Networks}.
\newblock John Wiley and Sons, 2007.

\bibitem{Ha1994}
Bruce Hajek.
\newblock A queue with periodic arrivals and constant service rate.
\newblock In {\em Probability, Statistics and Optimization—a Tribute to Peter
  Whittle}. Wiley New York, 1994.

\end{thebibliography}
\bibliographystyle{unsrt}
\end{document}